\documentclass[11pt,a4paper]{amsart}

\usepackage{amsfonts,amsmath,amssymb,amsthm}

\usepackage[utf8]{inputenc}
\usepackage[T1]{fontenc}

\usepackage{hyperref}
\usepackage{siunitx}

\newcommand{\Cc}{\mathbb{C}} 

\newcommand{\Rr}{\mathbb{R}}
\newcommand{\Nn}{\mathbb{N}}
\newcommand{\Zz}{\mathbb{Z}}
\newcommand{\Qq}{\mathbb{Q}}

\renewcommand {\le}{\leqslant}
\renewcommand {\ge}{\geqslant}
\renewcommand {\epsilon}{\varepsilon}

\newcommand{\defi}[1]{\emph{#1}}

\newcommand{\pattern}{\underline{\mathrm{m}}}

{\theoremstyle{plain}
	\newtheorem{theorem}{Theorem}[section]    
	\newtheorem{lemma}[theorem]{Lemma}       
	\newtheorem{proposition}[theorem]{Proposition}      
	\newtheorem{corollary}[theorem]{Corollary}      
	\newtheorem*{theorem*}{Theorem}
}
{\theoremstyle{remark}
	\newtheorem{example}[theorem]{Example}
	\newtheorem{exercise}[theorem]{Exercise}
}

\usepackage[a4paper]{geometry}
\makeatletter
\renewcommand{\pod}[1]{\allowbreak\mathchoice
	{\if@display \mkern 8mu\else \mkern 8mu\fi (#1)}
	{\if@display \mkern 8mu\else \mkern 8mu\fi (#1)}
	{\mkern4mu(#1)}
	{\mkern4mu(#1)}
}
\makeatother

\title{Around the gcd of the values of two polynomials}

\author{Arnaud Bodin}
\author{Christian Drouin}

\email{arnaud.bodin@univ-lille.fr}
\email{christian.drouin@wanadoo.fr}

\address{Université de Lille, CNRS, Laboratoire Paul Painlevé, 59000 Lille, France}
\address{Seignosse, France}



\date{\today}

\begin{document}

\begin{abstract}
We propose a mathematical walk around the gcd of the values $A(n)$ and $B(n)$ of two polynomials evaluated at an integer $n$. This is an opportunity to use a very powerful tool: the resultant.
\end{abstract}

\maketitle

\section{Motivation}




\subsection*{Two polynomials}

Consider the polynomials $A(x) = x^3-5x^2+10x-12$ and $B(x) = x^2+3$. For $n \in \Zz$, let's note $G(n) = \gcd( A(n), B(n))$, which we write down as $A(n) \wedge B(n)$. Here are the values of $G(n)$ for $n$ ranging from $0$ to $30$:
{
\arraycolsep=3pt
\[
\begin{array}{*{31}{c}}
3 & 2 & 1 & 12 & 1 & 2 & 3 & 52 & 1 & 6 & 1 & 4 & 3 & 2 & 1 & 12 & 1 & 2 & 3 & 4 & 13 & 6 & 1 & 4 & 3 & 2 & 1 & 12 & 1 & 2
\end{array}
\]
}%
Even if the first values don't suggest it, the sequence of $G(n)$ is periodic, but its period can be very large. Here, the sequence $(G(n))_{n\in\Zz}$ has period $156$, its largest element also being $156$.
How can we show that such a period exists, and how can we estimate it?
We're going to break down the study of the sequence of $G(n)$ into the study of several sequences of $G(n) \wedge p^\infty$ terms. We denote by $N \wedge p^\infty$ the greatest power $p^\omega$ dividing $N$ and denote by $\nu_p(N) = \omega$ the valuation of $N$. 
The Chinese Remainder Theorem will enable us to reconstitute the set $\{G(n)\}_{n\in\Zz}$.

Let's continue with the previous example. Here are the powers of $2$ that can be extracted from $G(n)$ for the first values $n=0,1,2,\ldots$ as above:
{
\arraycolsep=3pt
\[
\begin{array}{*{31}{c}}
1 & 2 & 1 & 4 &
1 & 2 & 1 & 4 &
1 & 2 & 1 & 4 &
1 & 2 & 1 & 4 & \ldots
\end{array}
\]
}%
A periodic pattern $[1,2,1,4]$ of length $4$ is clearly visible.
This is the same phenomenon for the prime numbers $3$ and $13$:
\[
\begin{array}{ll}
p = 2 &  \pattern_2 = [1,2,1,4] \\
p = 3 & \pattern_3 = [3,1,1] \\
p = 13 & \pattern_{13} = [1,1,1,1,1,1,1,13,1,1,1,1,1]
\end{array}
\]
For all other primes, $G(n) \wedge p = 1$.
In the general case, we'll explain how to reconstruct the values of $G(n)$ from the patterns, and explain what form the patterns can take.

\subsection*{Content of the paper}

First, we will use the resultant to prove that the sequence
$(G(n))_{n\in\Zz}$  is periodic and we will explain how it decomposes into its \emph{patterns} or basic components (Theorem \ref{th:pattern}).
We will point out strong constraints on these patterns 
(Theorem \ref{th:valpresultant}) and in some situations, provide a direct formula for them; this is the case if one of the polynomials is of degree one
(Section \ref{sec:degone}) or if the polynomials decompose into a product of distinct linear factors modulo $p$ (Section \ref{sec:simpleroots}).

\subsection*{Only one polynomial}

Let's start with the case of a single polynomial $A(x) = a_d x^d + \cdots + a_1 x + a_0 \in \Zz[x]$. Since for every monomial $a_k n^k$ of $A(n)$ we have $a_k(n+p^\alpha)^k \equiv a_k n^k \pmod {p^\alpha}$, then:
\begin{lemma}
    \label{lem:periodic}
    Let $p$ be a prime number and $\alpha \in \Nn$.
    The sequence of terms $A(n) \wedge p^\alpha$ is periodic with a period dividing $p^\alpha$.
\end{lemma}

It is worth noting that for every non-constant polynomial $A(x) \in \Zz[x]$, there are infinitely many primes $p$, such that $p$ divides $A(n)$ for some $n\in\Zz$ (see Schur \cite{SchurPoly1912}).
Note also that, when we study $A(n)$ modulo $p^\alpha$, we obtain a function $n \mapsto A(n) \pmod{p^\alpha}$ of $\Zz/ p^\alpha \Zz$ in itself. More generally, for $m$ fixed, there are $m^m$ different functions $f: \Zz/m\Zz \to \Zz/m\Zz$, but if we count only functions $A :  \Zz/m\Zz \to \Zz/m\Zz$ induced by polynomials $A \in \Zz[x]$ there are far fewer (their number is $\prod_{k=0}^{m} \frac{m}{\gcd(m,k!)}$, which can be demonstrated using the \emph{falling factorial}, see Bhargava \cite{Bh}).
For example, if $m=4$ there are $4^4 = 256$ functions $f : \Zz/4\Zz \to \Zz/4\Zz$ but only $64$ arise from a polynomial.
For our problem, we're interested in $f(n) \wedge 4$, which can only take the values $1$, $2$ or $4$.
When we count the number of possibilities for $[f(0) \wedge 4, f(1) \wedge 4, f(2) \wedge 4, f(3) \wedge 4]$, there are in theory $3^4 = 81$ possibilities, but in fact only $25$ come from polynomial functions.

Lemma \ref{lem:periodic} implies that $G(n) \wedge p^\alpha = A(n) \wedge B(n) \wedge p^\alpha$ is also periodic, with a period dividing $p^\alpha$.
This time, however, the sequence $(G(n))_{n\in\Zz}$ is periodic. Why is this so? Thanks to the resultant!

\section{Prime factors of the resultant}
\label{sec:resultant}

In this section, we show that the sequence $(G(n))_{n\in\Zz}$ is periodic and explain how it is decomposed using patterns.
Let $A(x) = a_dx^d+ \cdots +a_1x +a_0$ and $B(x) = b_ex^e+ \cdots + b_1x+b_0$ be two polynomials with coefficients in a field $k$, with $a_d \neq 0$ and $b_e \neq 0$.
The \defi{resultant} $\Delta = \det(S) \in k$ is the determinant of a $(d+e)\times(d+e)$ matrix $S$, called the \defi{Sylvester matrix}:
\[
\Delta = \det \begin{pmatrix}
 a_d    &        &        &        & b_e    &        &        \\
 \vdots & a_d    &        &        & \vdots & \ddots &        \\
 a_1    & \vdots & \ddots &        & \vdots &        & b_e    \\
 a_0    & a_1    &        & a_d    & b_1    &        & \vdots   \\
        & a_0    &        & \vdots & b_0    &        & \vdots \\
        &        & \ddots & a_1    &        & \ddots & b_1 \\
        &        &        & a_0    &        &        & b_0    \\
\end{pmatrix}
\]
The first $e$ columns are formed by the coefficients of $A(x)$ (with an offset at each column; zero coefficients are not indicated), the last $d$ columns are formed by the coefficients of $B(x)$.
The resultant is used to detect whether $A(x)$ and $B(x)$ have a common root. 
It is calculated using one of the following formulas:
\begin{theorem}
    \label{th:resultant}
    Let $\alpha_1,\ldots,\alpha_d$ be the roots of $A(x)$ in $\bar{k}$.
    Let $\beta_1,\ldots,\beta_e$ be the roots of $B(x)$ in $\bar{k}$.    
    Then 
    \[ \Delta = a_d^e b_e^d \prod_{\substack{1 \le i \le d \\ 1 \le j \le e}} (\alpha_i-\beta_j) = a_d^e \prod_{1 \le i \le d} B(\alpha_i). \]
\end{theorem}
Here $\bar{k}$ denotes an algebraic closure of $k$, e.g.{} if $k = \Rr$ then $\bar k = \Cc$.
For this result, and the next two, we refer to an algebra book, for example to \cite[Ch.\,4,\S\,8]{Lang}.
\begin{corollary}
There exists $x_0 \in \bar{k}$ such that $A(x_0) = 0$ and $B(x_0) = 0$ if and only if $\Delta = 0$.  
\end{corollary}

Let's discuss another property of the resultant in the case of polynomials with integer coefficients: Bézout's identity.
\begin{proposition}
    \label{prop:bezout}
    For $A(x) , B(x) \in \Zz[x]$ coprime polynomials (in $\Qq[x]$), there exists $U(x), V(x) \in \Zz[x]$ such that:
    \begin{equation}
        \label{eq:bezout}
        A(x)U(x) + B(x)V(x) = \Delta
    \end{equation}
    In addition, we can assume $\deg(U) < \deg(B)$ and $\deg(V) < \deg(A)$.
\end{proposition}

Such a Bézout's identity is first obtained on $\Qq$. Since $A(x)$ and $B(x)$ are coprime in $\Qq[x]$, there exists $U_0(x),V_0(x) \in \Qq[x]$ such that $A(x)U_0(x) + B(x)V_0(x) = 1$. Multiplying by the denominators of the coefficients of $U_0(x)$ and $V_0(x)$ gives an equation $A(x)U_1(x) + B(x)V_1(x) = r$, where $U_1(x), V_1(x) \in \Zz[x]$ and $r \in \Zz$.
We will explain in Section \ref{sec:constraints} (just after Proposition \ref{prop:resmodp}) why the resultant is one of the integers $r$ that can be obtained in this way.

\begin{corollary}
    \label{cor:factor}
    For all $n \in \Zz$, $G(n) | \Delta$.
\end{corollary}

\begin{proof}
    Thanks to this Bézout's identity, if $d | A(n)$ and $d | B(n)$ then $d | \Delta$.
\end{proof}

So the prime numbers $p$ that are factors of $G(n)$ are prime factors of the resultant $\Delta$. Of course, there are a finite number of such primes.
We'll see that the sequence $(G(n) \wedge p^\infty)_{n\in\Zz}$ is periodic.
The \defi{pattern} associated with the prime number $p$ is the list of elements of the sequence forming a minimal period:
\[ 
  \pattern_p = \left[ G(0) \wedge p^{\infty}, G(1) \wedge p^{\infty},\ldots,G(p^\mu-1) \wedge p^{\infty} \right]_{p^\mu}
\]
(The index to the right of the closing bracket indicates the length of the pattern.)
For $n\in \Zz$, we denote by $\pattern_p(n) = G(n) \wedge p^{\infty}$ the $n$-th term of the pattern extended by periodicity.

We group the first results in the following theorem.
\begin{theorem}
    \label{th:pattern}
    Let $A(x) , B(x) \in \Zz[x]$ be coprime polynomials (in $\Qq[x]$).
    	Let $G$ be defined by $G(n) = A(n) \wedge B(n)$, $n\in \Zz$.
    \begin{enumerate}
      \item The patterns are well defined: the sequence $(G(n) \wedge p^{\infty})_{n\in\Zz}$ is periodic, of a period dividing $p^{\omega_p}$ where $\omega_p = \nu_p(\Delta)$.
      
      \item \label{it:prod} For all $n \in \Zz$, $G(n) = \prod_{p | \Delta} \pattern_p(n)$.
      
      \item The sequence $( G(n) )_{n \in \Zz}$ is periodic, with a period dividing $\Delta$.
      
      \item $\{ G(n) \}_{n \in \Zz} = \prod_{p | \Delta} \{\pattern_p\} $
    \end{enumerate}
\end{theorem}

Remarks on each item:
\begin{enumerate}
  \item Recall that we noted $\omega_p = \nu_p(\Delta)$ as the $p$-valuation of the resultant, i.e.{} $p^{\omega_p}$ is the highest power of $p$ that divides $\Delta$.

  \item The second point simply states that the patterns correspond to the decomposition into prime factors of $G(n)$.

  \item Here we find a result by Frenkel--Pelik\'{a}n \cite{FP} (see also Frenkel--Z\'{a}br\'{a}di \cite{FZ}) and Bodin--Dèbes--Najib \cite{BDN20}.

  \item Let's explain how the last point differs from the second.
  The second point proves that $G(n)$ is the product of $\pattern_{p_i}(n)$, where $n$ is the same integer for each prime number $p_i$. 
  The fourth point proves that if we take any element $\pattern_{p_1}(n_1)$ of the pattern $\pattern_{p_1}$, any element $\pattern_{p_2}(n_2)$ of the pattern $\pattern_{p_2}$, \ldots{} then there exists $n\in\Zz$ such that $G(n)$ is equal to the product of $\pattern_{p_i}(n_i)$.
\end{enumerate}

\begin{example}
Let $A(x) = (x-5)(x-27)$ and $B(x) = x^2+3x+9$.
The resultant is $\Delta = \num{40131} = 3^2 \times 7^3 \times 13$.
The patterns associated with the prime factors are:
\[
\begin{array}{ll}
p = 3 &  \pattern_3 = [9,1,1,3,1,1,3,1,1]_{9} \\
p = 7 & \pattern_7 = [1, 1, 1, 1, 1, 49, 7, 1, 1, 1, 1, 1, 7, 7, 1,\ldots]_{49} \\
p = 13 & \pattern_{13} = [1,13,1,1,1,1,1,1,1,1,1,1,1]_{13}
\end{array}
\]
Here are the values of $G(n) = A(n) \wedge B(n)$ for $n=0,\ldots,30$:
{
\arraycolsep=3pt
\[
\begin{array}{*{32}{c}}
9&13&1&3&1&49&21&1&1&9&1&1&21&7&13&3&1&1&9&7&7&3&1&1&3&1&7&819&1&1&3&
\ldots
\end{array}
\]
}
The sequence $(G(n))_{n\in\Zz}$ is periodic, its period is $\num{5733} = 3^2 \times 7^2 \times 13$ (it's the product of the pattern lengths).
The set of possible values for $G(n)$ is:
\[ \big\{ 
  1, 3, 7, 9, 13, 21, 39, 49, 63, 91, 117, 147, 273, 441, 637, 819, 1911, 5733
\big\}, \]
which is exactly the product of the pattern values:
\[
    \{ \pattern_3 \} \times \{ \pattern_7 \} \times \{ \pattern_{13} \} 
  = \{1, 3, 9 \} \times \{ 1, 7 , 49 \} \times \{ 1, 13 \}.
\]
\end{example}

\begin{proof}
~
\begin{enumerate}
    \item Corollary \ref{cor:factor} proves that $G(n) | \Delta$, so $\nu_p(G(n)) \le \omega_p$ and $G(n) \wedge p^\infty = G(n) \wedge p^{\omega_p}$.
    By Lemma \ref{lem:periodic} the sequence $(G(n) \wedge p^{\omega_p})_{n\in\Zz}$ is periodic, of minimal period of the form $p^\mu$ with $\mu \le \omega_p$.
    
    \item Once again, Corollary \ref{cor:factor} proves that the prime factors of $G(n)$ are the only prime factors of the resultant.

    \item Since the patterns are periodic and there are only a finite number of prime numbers $p$ to consider then the period is smaller than the product of the periods, so a divisor of $\prod_{p | \Delta} p^{\omega_p} = \Delta$.

    \item The point \eqref{it:prod} proves inclusion $\subset$.
    For the other inclusion, we need to check that all products of pattern components are feasible.
    Let $p_i$ be a prime divisor of the resultant and let $m_i$ be an element of the pattern $\pattern_{p_i}$, $i=1,\ldots,\ell$. By definition, there exists $n_i \in \Zz$ such that $G(n_i) \wedge p^{\omega_{i}} = m_i$. By the Chinese Remainder Theorem, there exists $n \in \Zz$ such that $n \equiv n_i \pmod{p^{\omega_i}}$ for all $i=1,\ldots,\ell$.
    This integer $n$ verifies $G(n) \equiv G(n_i) \pmod{p^{\omega_i}}$, 
    and since $G(n) | \Delta$ it implies $G(n) \wedge p^{\omega_i} = m_i$, $i=1,\ldots,\ell$.
\end{enumerate}
\end{proof}

\section{Constraints on patterns}
\label{sec:constraints}

We will now present strong constraints on the structure of the patterns.
In this section we assume that the polynomials $A(x)$ and $B(x)$ are monic (their dominant coefficient is $1$). For each prime number $p$, we investigate the shape of the associated pattern.
Let's start with a very common case, which is a slightly more general version of a result by Frenkel--Pelik\'{a}n \cite[Theorem 6]{FP}:
\begin{proposition}
    \label{prop:valpresone}
    If $\nu_p(\Delta) = 1$ then the pattern associated with $p$ is $[p,1,1,\ldots,1]_p$ (up to permutation).
\end{proposition}

Proposition \ref{prop:valpresone} is a consequence of a more general result that imposes many constraints on the patterns that can be realized.

\begin{theorem}
    \label{th:valpresultant}
      Let $A(x) , B(x) \in \Zz[x]$ be coprime monic polynomials, let $\Delta$ be their resultant and let $G$ be defined by $G(n) = A(n) \wedge B(n)$, $n\in \Zz$.    
    If an integer $q_i$ divides $G(n_i)$ for $i=1,\ldots,\ell$, with $\ell \le \deg(A)+\deg(B)$, then $q_1 q_2 \cdots q_\ell$ divides
    $\Delta \times \prod_{1 \le i < j \le \ell} (n_j-n_i)$.
    In particular, if $p^{\omega_1}$ divides $G(n_1)$ and $p^{\omega_2}$ divides $G(n_2)$ then $$\nu_p(n_2-n_1) \ge \omega_1+\omega_2 - \nu_p(\Delta).$$
\end{theorem}

\begin{example}
Consider $p=5$ and $\nu_p(\Delta)=2$, then the possible patterns are:
\begin{center}
\begin{tabular}{l}
$[1]_1$ \quad \\
$[5,1,1,1,1]_5$ \quad up to permutation \\
$[5,5,1,1,1]_5$ \quad up to permutation \\
$[25,1,1,1,1,5,1,1,1,1,5,1,1,1,\ldots]_{25}$ \quad up to circular permutation \\
\end{tabular}
\end{center}

We leave it to the reader to find monic polynomials that realize these patterns!

%
%
%

Theorem \ref{th:valpresultant} shows that the other patterns are not realized.
Since $\nu_5(\Delta)=2$ then the only elements making up the pattern are $1$, $5$ or $25$. 
For example, the pattern $[25,1,1,1,1]_5$ cannot be realized.
Indeed, if $5^2 | G(n_1)$ and $5^2 | G(n_2)$, Theorem \ref{th:valpresultant} with $\omega_1=\omega_2=2$ implies that $\nu_5(n_2-n_1) \ge 2$, hence $n_2 \equiv n_1 \pmod{25}$ and prevents the pattern in question from being realized.
The pattern $[25,5,1,1,1,5,5,1,1,1,5,5,1,1,\ldots]_{25}$ is similarly excluded by setting $\omega_1=2$ and $\omega_2=1$.
More generally $5$ consecutive elements never include both $25$ and $5$.

It's also easy to show that out of $5$ consecutive elements of a pattern, at most two are divisible by $5$.
If $\deg(A)=1$ and $\deg(B)=1$ then $A(x)=x-r_0$ and $B(x)=x-s_0$ have at most one common root modulo $5$. In the case $\deg(A) > 1$ or $\deg(B)>1$
and if $5 | G(n_i)$, $i=1,2,3$, then Theorem \ref{th:valpresultant} gives the inequality
$$\nu_5\big( (n_2-n_1)(n_3-n_1)(n_3-n_2) \big) \ge 1,$$
so $\nu_5(n_j-n_i) \ge 1$ for a certain pair $(i,j)$, so for instance $n_2 \equiv n_1 \pmod{5}$.
This excludes patterns $[5,5,5,1,1]_5$ or $[5]_1$, for example.
\end{example}

\begin{proof}[Proof of Theorem \ref{th:valpresultant}]
    Consider the row vector $X = (n^{d+e-1}, \ldots, n^2,n,1)$.
    Multiply $X$ to the right of the Sylvester matrix $S$, then
    $$X \times S = \big(n^{e-1}A(n), \ldots, nA(n), A(n), n^{d-1}B(n),\ldots,nB(n),B(n) \big)$$
    Consider the matrix $V$:
    $$V = 
    \begin{pmatrix}
    n_1^{d+e-1} & \cdots & \cdots & n_1^2 & n_1 & 1 \\
    n_2^{d+e-1} & \cdots & \cdots & n_2^2 & n_2 & 1 \\   
    \vdots & & & & & \\
    n_l^{d+e-1} & \cdots & \cdots & n_l^2 & n_l & 1 \\      
    1 & 0 & \cdots & & & \\
    0 & 1 & 0 & \cdots & & \\
    \cdots & & & & & \\
    \end{pmatrix}$$
    $V$ is of size $(d+e)\times(d+e)$,
    the first $\ell$ rows are of the form $(n_i^{d+e-1}, \ldots, n_i^2,n_i,1)$.
    The following rows contain a single $1$ and form an identity sub-matrix at bottom left.
    \begin{itemize}
        \item The determinant of $V$ is calculated as a Vandermonde determinant of size $\ell \times \ell$:
        \[ \det V = \pm \prod_{1 \le i < j \le \ell} (n_j-n_i). \]

        \item By definition $\det S = \Delta$.

        \item The first $\ell$ rows of $V \times S$, are of the form 
        \[\big(n_i^{e-1}A(n_i), \ldots, n_iA(n_i), A(n_i), n_i^{d-1}B(n_i),\ldots,n_iB(n_i),B(n_i) \big) \]
        as explained above.
        Thus, if $q_i$ divides $G(n_i)$, then $q_i$ divides $A(n_i)$ and $B(n_i)$ so $q_i$ divides all the elements in row $i$ of the matrix $V \times S$, for $i=1,\ldots,\ell$. So $q_1q_2\cdots q_\ell$ is a factor of $\det(V \times S)$.
        This proves that $q_1q_2\cdots q_\ell$ divides $\Delta \times \prod_{1 \le i < j \le \ell} (n_j-n_i)$.
    \end{itemize}

\end{proof}

We need a very useful result by Gomez--Gutierrez \cite{GG} (see also the proof of \cite[Theorem 6]{FP}) which provides an inequality between the degree of the gcd of two polynomials modulo $p$ and the valuation in $p$ of the resultant of these two polynomials. 
\begin{proposition}
    \label{prop:resmodp}
    Let $A(x), B(x) \in \Zz[x]$ be monic polynomials.
    Let $D(x)$ be the gcd of $A(x)$ and $B(x)$ modulo $p$.
    Then $\deg(D) \le \nu_p(\Delta)$.
\end{proposition}

We're going to describe a simple proof, which will provide an opportunity to present the resultant via linear maps. 
Let $k_n[x]$ be the vector space of polynomials of degree $\le n$, and let choose for this vector space of dimension $n+1$ the (reverse) canonical basis $(x^n,x^{n-1},\ldots,x,1)$.
Let $A(x) \in k_d[x]$ and $B(x) \in k_e[x]$ (not necessarily monic).
Consider the linear map $\varphi$ defined by:
\[ \begin{array}{rccc}
\varphi : & k_{e-1}[x] \times k_{d-1}[x] & \longrightarrow & k_{e+d-1}[x] \\
          &                            (U,V) & \longmapsto     & AU+BV \\
\end{array}
\]

The matrix of $\varphi$ associated with the canonical bases is exactly the Sylvester matrix $S$ of the polynomials $A(x)$ and $B(x)$.
If $A(x)$ and $B(x)$ are coprime, then there exist $U(x) \in k_{e-1}[x]$ and $V(x) \in k_{d-1}[x]$ such that $A(x)U(x)+B(x)V(x)=1$, which implies that $\varphi$ is surjective, so by dimensional reasons, bijective, and so in this case $\Delta = \det(S) \neq 0$.
If $A(x)$ and $B(x)$ are not coprime, then there exists $D(x) \in k[x]$ such that $A = DA_0$, $B=DB_0$ with $\deg(A_0) < \deg(A)$ and $\deg(B_0) < \deg(B)$. The relation $AB_0 - BA_0 = 0$, proves that $\varphi(B_0,-A_0) = 0$, so $\varphi$ is not injective and in this case $\Delta = 0$.

\medskip

We are going to use these considerations of linear algebra to prove Proposition \ref{prop:resmodp}, but first we complete our explanations on Proposition \ref{prop:bezout}.
Let $A(x), B(x) \in \Zz[x]$ be coprime polynomials (in $\Qq[x]$), by Bézout's identity there exist $U_0(x),V_0(x) \in \Qq[x]$ such that $A(x)U_0(x) + B(x)V_0(x) = 1$, which implies $\varphi(U_0,V_0) =1$. So that in terms of matrices $S \times W_0 = E$, where $W_0$ is the column vector associated with $(U_0,V_0)$ and $E$ is the column vector $(0,0,\ldots,0,1)$. 
We now explain how to find polynomials $U_1(x),V_1(x)$ with integer coefficients, such that $A(x)U_1(x) + B(x)V_1(x) = \Delta$ that is to say $\varphi(U_1,V_1) = \Delta$.
The inverse $S^{-1}$ of $S$ can be computed by $S^{-1} = \frac{1}{\det(S)} S^*$ where $S^*$ denotes the transpose of the cofactor matrix of $S$.
Then $\det(S)\,I=SS^{*}$ ($I$ being the identity matrix) and $\det(S)\,E=SS^{*}E$ ($E$ being the column vector $(0,0,\ldots,0,1)$).
Let $W_1 = S^{*}E$, this is a column vector with integer coefficients. 
We denote by $U_1(x) \in \Zz[x]$ (resp.~$V_1(x) \in \Zz[x]$) the polynomial whose coefficients are the $e$ first (resp.~$d$ last) components of $W_1$.
As $SW_1 = \Delta E$, we get $\varphi(U_1,V_1) = \Delta$ so that $A(x)U_1(x) + B(x)V_1(x) = \Delta$.



\begin{proof}[Proof of Proposition \ref{prop:resmodp}]
    Let $D(x)$ be the gcd of $A(x)$ and $B(x)$ modulo $p$, i.e.{} we can write $A(x) \equiv D(x) A_0(x) \pmod{p}$ and $B(x) \equiv D(x) B_0(x) \pmod{p}$ with $A_0(x)$ and $B_0(x)$ monic polynomials, with integer coefficients, with no common factors modulo $p$. Let $\ell = \deg(D) \ge 1$. 
    Let $A_0(x) = \alpha_{d-\ell}x^{d-\ell} + \cdots +\alpha_0$, $B_0(x) = \beta_{e-\ell}x^{e-\ell}+\cdots+\beta_0$. Denote by $W_0$ the vector corresponding to the pair of polynomials $(B_0(x),-A_0(x)) \in \Rr_{e-1}[x] \times \Rr_{d-1}[x]$, and more generally $W_i$ the vector corresponding to the pair of polynomials $(x^i B_0(x),-x^i A_0(x)) \in \Rr_{e-1}[x] \times \Rr_{d-1}[x]$, for $i=0,\ldots,\ell-1$:
    \[
    W_0 = \left(\begin{smallmatrix} 
    0 \\ \vdots \\ 0 \\ \beta_{e-\ell} \\ \vdots \\ \beta_0 \\ 
    0 \\ \vdots \\ 0 \\ -\alpha_{d-\ell} \\ \vdots \\ -\alpha_0 
    \end{smallmatrix}\right)
    \qquad \cdots \qquad
    W_{\ell-1} = \left(\begin{smallmatrix} 
    \beta_{e-\ell} \\ \vdots \\ \beta_0 \\ 0 \\ \vdots \\ 0 \\ 
    -\alpha_{d-\ell} \\ \vdots \\ -\alpha_0 \\ 0 \\ \vdots \\ 0
    \end{smallmatrix}\right)
    \]
    Since $A(x) B_0(x) - B(x)A_0(x) \equiv 0 \pmod{p}$, then $\varphi(-B_0,A_0) = S W_0$ is a vector whose coefficients are all divisible by $p$. And likewise
    $A(x) (x^iB_0(x)) - B(x)(x^iA_0(x)) \equiv 0 \pmod{p}$, so $S W_i$ has all its coefficients divisible by $p$, for $i=0,\ldots,\ell-1$.
    
    Denote by $W$ the matrix $(d+e) \times (d+e)$ whose $\ell$ first columns are formed by $W_{\ell-1},W_{\ell-2},\ldots,W_0$, completed by an identity block at the bottom right (zero coefficients are not indicated):
    \[ W = \left(\begin{smallmatrix}
    \beta_{e-\ell}   &               &        &               &   &   &        &        &        &        &        &   \\
    \vdots        & \beta_{e-\ell}   &        &               &   &   &        &        &        &        &        &   \\
    \beta_0       & \vdots        & \ddots &               &   &   &        &        &        &        &        &   \\
                  & \beta_0       &        & \beta_{e-\ell}   &   &   &        &        &        &        &        &   \\
                  &               & \ddots & \vdots        & 1 &   &        &        &        &        &        &   \\
                  &               &        & \beta_0       &   & 1 &        &        &        &        &        &   \\
    -\alpha_{d-\ell} &               &        &               &   &   & \ddots &        &        &        &        &   \\  
    \vdots        & -\alpha_{d-\ell} &        &               &   &   &        & \ddots &        &        &        &   \\
    -\alpha_0     & \vdots        & \ddots &               &   &   &        &        & \ddots &        &        &   \\
                  & -\alpha_0     &        & -\alpha_{d-\ell} &   &   &        &        &        & \ddots &        &   \\
                  &               & \ddots & \vdots        &   &   &        &        &        &        & \ddots &   \\
                  &               &        & -\alpha_0     &   &   &        &        &        &        &        & 1 \\
    \end{smallmatrix}\right)\]
    The first $\ell$ columns of $SW$ are divisible by $p$. So $p^\ell | \det(SW)$. Since we're assuming $B$ and $D$ as monic, then $\beta_{e-\ell}=1$ and therefore $\det(W) = 1$. Hence $p^\ell | \det(S) = \Delta$, so $\deg(D) \le \nu_p(\Delta)$.
\end{proof}

\begin{proof} [Proof of Proposition \ref{prop:valpresone}]
By hypothesis $\nu_p(\Delta) = 1$, so $A(x)$ and $B(x)$ have a common factor modulo $p$, which by Proposition \ref{prop:resmodp} is necessarily of degree $1$. Thus, there exists $n_1 \in \Zz$ such that $A(n_1) \equiv 0 \pmod{p}$ and $B(n_1) \equiv 0 \pmod{p}$.
Thus $p | G(n_1)$. By Corollary \ref{cor:factor}, we know that $p^2$ does not divide $G(n_1)$. 
Let $n_2$ be such that $p | G(n_2)$ then, by Theorem \ref{th:valpresultant} applied with $\omega = 1$, we have $\nu_p(n_2)-\nu(n_1) > 0$, so $n_2 \equiv n_1 \pmod{p}$. Thus, for a pattern of length $p$, the term $p$ can appear here only once: $\pattern_p = [p,1,\ldots,1]_p$ up to permutation.
\end{proof}

\section{Case of a polynomial of degree $1$}
\label{sec:degone}

Let $A(x) = a_1x+a_0 \in \Zz[x]$ a degree $1$ polynomial (not necessarily monic).
First notice that if $A(n) \not\equiv 0 \pmod{p}$ for any $n \in \Zz$, then $A(n) \wedge p = 1$ and it implies $G(n)=1$ for any $n \in \Zz$.
We first provide an explicit formula in the simple case where the degree of $A$ is $1$, see Drouin \cite{drouin}.

\begin{proposition}
    \label{prop:deg1}
    Let $A(x) = a_1x+a_0$ with $a_0 \wedge a_1 = 1$. 
    Let $B(x) = x^e + b_{e-1} x^{e-1} + \cdots + b_0$ be a monic polynomial, coprime with $A(x)$.
    Let $\omega = \nu_p(\Delta)$.
    The pattern $\pattern_p$ defined by $A(n) \wedge B(n) \wedge p^{\omega}$
    is the basic pattern $[n \wedge p^\omega]_{p^\omega}$ up to circular permutation.
\end{proposition}

In cases where $a_0 \wedge a_1 \neq 1$ or $B(x)$ is not monic, the result may not always be so simple, see \cite{drouin}.

\begin{proof}
    Let $\alpha = - \frac{a_0}{a_1}$ be the root of $A(x) = a_1x+a_0$.
    By the second equality of Theorem \ref{th:resultant} then
    \[ \Delta 
    = a_1^e B\left( - \frac{a_0}{a_1} \right) \]
    So
    \begin{equation}
    \label{eq:deg1}
    \Delta = (-a_0)^e + b_{e-1}a_1(-a_0)^{e-1} + \cdots + b_ka_1^{e-k}(-a_0)^{k} + \cdots + b_1 a_1^{e-1}(-a_0) + b_0a_1^e
    \end{equation}

    Let $p$ be a prime factor of $\Delta$.
    Let's prove that $p$ does not divide $a_1$: indeed, if $p | a_1$ then by \eqref{eq:deg1} we would have $p | a_0$, which contradicts $a_0 \wedge a_1 = 1$.
    So $a_1$ is invertible modulo $p$ and therefore invertible modulo the powers of $p$.
    Let's denote $\omega = \nu_p(\Delta)$ and $\overline{a_1} \in \Zz$ an inverse of $a_1$ modulo $p^\omega$.
    Set $\tilde \alpha = -a_0 \overline{a_1}$.
    Write an integer $n \in \Zz$ in the form $n = \tilde\alpha + m p^k$ (with $m$ not divisible by $p$).
    On the one hand 
    \[ A(n) 
    = A(\tilde \alpha + mp^k) 
    \equiv a_1(-a_0 \overline{a_1} + mp^k) + a_0
    \equiv mp^k \pmod{p^\omega},
     \]
     so $A(n) \wedge p^\omega = p^k \wedge p^\omega = p^{\min(k,\omega)}$.
     On the other hand $B(n) = B(\tilde\alpha + mp^k) \equiv B(-a_0\overline{a_1}) \pmod{p^k}$.
     By the integer equation \eqref{eq:deg1},
     $a_1^e B\left( -a_0 \overline{a_1} \right)  \equiv \Delta \pmod{p^\omega}$.
     Since $p$ does not divide $a_1$ and $p^\omega$ divides $\Delta$ then $p^\omega$ divides $B\left( -a_0 \overline{a_1} \right)$.
     Thus $B(n) \wedge p^\omega = p^k \wedge p^\omega$.
     Finally $A(n) \wedge B(n) \wedge p^\omega = p^k \wedge p^\omega$, which corresponds exactly to the pattern $(n-\tilde\alpha) \wedge p^\omega$.
\end{proof}

\section{Case of split polynomials with simple roots modulo $p$}
\label{sec:simpleroots}

We provide a direct formula to compute the gcd of $A(n)$ and $B(n)$ in the case of polynomials that are split into distinct linear factors modulo $p$.
Consider a polynomial $A(x)$ with a simple root $\rho$ modulo $p$, i.e.:
\[ A(\rho) \equiv 0 \pmod{p} \qquad \text{ and } \qquad A'(\rho) \not\equiv 0 \pmod{p}\]
Hensel's Lemma allows us to ``uplift'' this root modulo $p^2$, $p^3$,\ldots.

\begin{theorem}[Hensel's Lemma]
For any $\omega > 0$, there exists $r \in \Zz$, such that $r \equiv \rho \pmod{p}$ and
$A(r) \equiv 0 \pmod {p^\omega}$.
\end{theorem}

The idea behind the proof is a variation of Newton's method for root approximation. We refer to \cite[Section 2.6]{NZM} for details. The proof is done by induction on $\omega$, the first step is to write a Taylor expansion around the root:
\[ A(\rho+hp) \equiv A(\rho) + hpA'(\rho) \pmod{p^2}. \]
Denoting $\overline{A'(\rho)} \in \Zz$ an inverse of $A'(\rho)$ modulo $p$ and setting
\[ h_0  = - \frac{A(\rho)}{p} \overline{A'(\rho)} \]
which make sense since $p$ divides $A(\rho)$,
then $A(\rho+h_0p) \equiv 0 \pmod{p^2}$. Thus $r = \rho + h_0p$ is a root modulo $p^2$. 

\medskip

Consider a monic polynomial $A(x)$ that is split and has simple roots modulo $p$, i.e.{} $A(x) \equiv (x-\rho_1)(x-\rho_2)\cdots(x-\rho_d) \pmod{p}$
where $\rho_i$ are pairwise distinct modulo $p$.
Then a variant of Hensel's Lemma allows us to factor $A(x)$ modulo any power of $p$. There exists $r_1,\ldots,r_d \in \Zz$ such that $r_i \equiv \rho_i \pmod{p}$ (and therefore $r_i \not\equiv r_j \pmod{p}$) with:
\[ A(x) \equiv (x-r_1)(x-r_2)\cdots(x-r_d) \pmod{p^\omega}. \]

Let $B(x)$ be another split monic polynomial with simple roots modulo $p$, $B(x) \equiv (x-\sigma_1)(x-\sigma_2)\cdots(x-\sigma_e) \pmod{p}$ and its factorization modulo $p^\omega$, $B(x) \equiv (x-s_1)(x-s_2)\cdots(x-s_e) \pmod{p^\omega}$.
Here's a direct formula for computing the gcd of the values.
\begin{theorem}
    \label{th:simpleroots}
    Let $A(x), B(x) \in \Zz[x]$ be monic coprime polynomials, such that both $A(x)$ and $B(x)$ are split and have simple roots modulo $p$.
    Let $p^\omega$ be the largest possible factor among all the $A(n) \wedge B(n)$.
    Let $n \in \Zz$. If there are $1 \le i \le d$ and $1 \le j \le e$ such that $n \equiv r_i \equiv s_j \pmod{p}$ then 
    \[ A(n) \wedge B(n) \wedge p^\omega = (n-r_i) \wedge (r_i - s_j) \wedge p^\omega. \]
    Otherwise $A(n) \wedge B(n) \wedge p^\omega = 1$.
\end{theorem}

\begin{proof}
    Let's fix $n \in \Zz$.
    For $A(n) \wedge B(n) \wedge p^\omega$ to be different from $1$ we need $A(n) \equiv 0 \pmod p$ and $B(n) \equiv 0 \pmod{p}$, so there exists $1 \le i_0 \le d$ such that $n \equiv r_{i_0} \pmod{p}$ and there exists $1 \le j_0 \le e$ such that $n \equiv s_{j_0} \pmod p$. Moreover, such $i_0$ and $j_0$ are unique because $r_i$ are pairwise distinct modulo $p$ and so are the $s_j$. 
    In other words, in the product $A(n) \equiv (n-r_1)(n-r_2)\cdots(n-r_d) \pmod{p^\omega}$ only the term $n-r_{i_0}$ is divisible by $p$ and in the product $B(n) \equiv (n-s_1)(n-s_2)\cdots(n-s_e) \pmod{p^\omega}$ only the term $n-s_{j_0}$ is divisible by $p$. Thus
    $A(n) \wedge B(n) \wedge p^\omega = (n-r_{i_0}) \wedge (n-s_{j_0}) \wedge p^\omega$, as $\gcd(a,b) = \gcd(a,b-a)$ then we also have $A(n) \wedge B(n) \wedge p^\omega = (n-r_{i_0}) \wedge (r_{i_0}-s_{j_0}) \wedge p^\omega$.
\end{proof}

\section{Better than the resultant?}
\label{sec:delta}

The resultant is not always the smallest integer that satisfies a Bézout identity. For example, with $A(x) = x^2+4$ and $B(x) = x^2-4$, the resultant is $\Delta=64$, but a smaller integer is obtained by Bézout's identity $A(x) \times 1 + B(x) \times (-1) = 8$.
We will denote by $\delta$ the smallest positive integer such that there exists $U(x),V(x) \in \Zz[x]$ with $A(x)U(x) + B(x)V(x) = \delta$.
As before, if $d | A(n)$ and $d | B(n)$ then $d | \delta$.
So, for any $n \in \Zz$, $G(n) = A(n) \wedge B(n)$ divides $\delta$.
Since the resultant also verifies such a Bézout identity (see Formula \eqref{eq:bezout}) then $\delta | \Delta$.

Here's a link between $\delta$ and the existence of common roots of $A(x)$ and $B(x)$ modulo powers of $p$. 
\begin{proposition}
    \label{prop:delta}
    Let $A(x), B(x) \in \Zz[x]$ be monic polynomials, coprime (in $\Qq[x]$).
    Assume that $A(x)$ and $B(x)$ are split and have simple roots modulo $p$.
    Then $\nu_p(\delta)$ is the largest integer $\mu$ such that there exists $n \in \Zz$ with $A(n) \equiv 0 \pmod{p^\mu}$ and $B(n) \equiv 0 \pmod{p^\mu}$.
    In particular, $\nu_p(\delta)$ is the largest exponent appearing in the pattern $\pattern_p$, which has length $p^{\nu_p(\delta)}$.
\end{proposition}

\begin{example}
    \label{ex:delta}
    Let $A(x) = x^2 - 9x + 16$ and $B(x) = x^2 - 7x + 12$.
    The resultant is $\Delta = 8$. Modulo $p=2$, $A(x) \equiv x(x-1)$ and $B(x) \equiv x(x-1)$ are split with simple roots.
    The common roots modulo $2$, are $0$ and $1$.
    Modulo $4$, the only common root is $n_0=0$: $A(0) \equiv 0 \pmod{4}$ and $B(0) \equiv 0 \pmod{4}$. Modulo $8$, $A(x)$ and $B(x)$ no longer have common roots.
    So Proposition \ref{prop:delta} gives us $\delta = 4$.
\end{example}

This result would no longer be valid if $A(x)$ or $B(x)$ had multiple factors.
There is a generalization by Taixés--Wiese  \cite[Corollary 2.12 (c)]{TW} in which the split hypothesis is no longer necessary, but there must be no multiple factors.
Another way of computing $\delta$ is due to Ayad \cite[Exercise 2.13]{ayad}
which we explain briefly: let $U(x), V(x) \in \Zz[x]$ be Bézout coefficients provided by the extended Euclidean algorithm such that $A(x)U(v) + B(x) V(x) = \Delta$. Let $c(U)$ be the \defi{content} of $U$, i.e.{} the gcd of the coefficients of $U(x)$, and $c(V)$ the content of $V(x)$. Then $\delta = \frac{\Delta}{\gcd(c(U),c(V))}$. Thus we obtain a Bézout identity for $\delta$ by starting from a Bézout identity for $\Delta$ and dividing by the gcd of all the coefficients of $U$ and $V$. 
For example, with the polynomials $A(x)$ and $B(x)$ of Example \ref{ex:delta}, with $U(x) = 2x-10$ and $V(x) = -2x+14$ we obtain Bézout's identity $A(x)U(x)+B(x)V(x) = 8$ which gives the resultant, but as the coefficients of $U(x)$ and $V(x)$ are all divisible by $2$, we easily obtain a Bézout identity giving $\delta=4$.

\begin{proof}[Proof of Proposition \ref{prop:delta}]
Let $\mu$ be the largest integer such that $A(n)$ and $B(n)$ have a common root modulo $p^\mu$. There therefore exists $n_0 \in \Zz$ such that $A(n_0) \equiv 0 \pmod{p^\mu}$ and $B(n_0) \equiv 0 \pmod{p^\mu}$. Bézout's identity $A(x)U(x)+B(x)V(x) = \delta$ applied to $x=n_0$ proves that $\delta \equiv 0 \pmod{p^\mu}$ and therefore $\nu_p(\delta) \ge \mu$.

\medskip

By contradiction,  assume that $\nu_p(\delta) > \mu$.
As in Section \ref{sec:simpleroots}, write $A(x) \equiv (x-\rho_1)(x-\rho_2)\cdots(x-\rho_d) \pmod{p}$ where the $\rho_i$ are pairwise distinct modulo $p$. This factorization is lifted by Hensel's Lemma modulo $p^{\mu+1}$ to $A(x) \equiv (x-r_1)(x-r_2)\cdots(x-r_d) \pmod{p^{\mu+1}}$ with $r_i \equiv \rho_i \pmod{p}$. The same applies to $B(x) \equiv (x-\sigma_1)(x-\sigma_2)\cdots(x-\sigma_e) \pmod{p}$ and its factorization modulo $p^{\mu+1}$, $B(x) \equiv (x-s_1)(x-s_2)\cdots(x-s_e) \pmod{p^{\mu+1}}$.

Bézout's identity on $\Zz$ is written $A(x)U(x)+B(x)V(x)=\delta$ where we take care to choose $\deg(U) < \deg(B)$ and $\deg(V) < \deg(A)$.
We evaluate this identity at $x=r_i$, as $A(r_i) \equiv 0 \pmod{p^{\mu+1}}$ and $\delta \equiv 0 \pmod{p^{\mu+1}}$ (because $\nu_p(\delta) > \mu$) then $B(r_i)V(r_i) \equiv 0 \pmod{p^{\mu+1}}$. But by definition of $\mu$, $A(x)$ and $B(x)$ have no common roots modulo $p^{\mu+1}$, so $B(r_i) \not\equiv 0 \pmod{p^{\mu+1}}$ and so $V(r_i) \equiv 0 \pmod{p}$ (in other words $p^{\mu+1}$ divides $B(r_i)V(r_i)$ but not $B(r_i)$ so $p$ divides $V(r_i)$).
This is true for each root $r_i$ of $A$, $i=1,\ldots,d$ and as $r_i \equiv \rho_i \pmod{p}$, then 
$V(\rho_i) \equiv 0 \pmod{p}$, $i=1,\ldots,d$.
We found $d$ roots to the polynomial $V(x)$ of degree $< d$ in the UFD ring $\Zz/p\Zz[x]$, so $V(x)$ is the zero polynomial modulo $p$. Bézout's identity modulo $p$ becomes 
$A(x)U(x) \equiv \delta \pmod{p}$, which is impossible for reasons of degree in $\Zz/p\Zz[x]$.   

\medskip

Let $\mu = \nu_p(\delta)$ be the largest integer such that $A(n)$ and $B(n)$ have a common root modulo $p^\mu$. For this common root $n_0$, we have $A(n_0) \wedge B(n_0) \wedge p^\infty = p^\mu$, and since for any $n$, $A(n) \wedge B(n)$ divides $\delta$, then $p^\mu$ is indeed the largest element of the pattern $\pattern_p$.

We now need to prove that the length of the pattern is $p^\mu$.
First of all, by Lemma \ref{lem:periodic} we know that this length divides $p^{\mu}$. 
To simplify the end of the proof, we assume that $n_0 = 0$, i.e. $A(x) \equiv x (x-r_2)\cdots(x-r_d) \pmod{p^\mu}$ with $r_1=0$ and $r_i \not\equiv r_j \pmod{p}$ (if $i \neq j$) and $B(x) \equiv x(x-s_2)\cdots(x-s_e) \pmod{p^\mu}$ with $s_1=0$ and $s_i \not\equiv s_j \pmod{p}$ (if $i \neq j$).
Then for $k=1,\ldots,\mu-1$, $A(p^k) \wedge B(p^k) \wedge p^\infty = p^k \wedge p^\mu = p^k$ which means that the pattern must be longer than or equal to $p^\mu$.  
\end{proof}

\begin{exercise}
	\label{ex:exo}
	Let \( A(x) = x^a + 1 \) and \( B(x) = x^b + 1 \). The goal of the exercise is to show that when \( A(x) \) and \( B(x) \) are coprime polynomials, the sequence \( \big( \gcd(A(n), B(n)) \big)_{n \in \mathbb{Z}} \) is a periodic sequence with pattern \([1, 2]\).	
	\begin{enumerate}
		\item What is the remainder of the Euclidean division of \( x^a + 1 \) by \( x + 1 \) (in \(\mathbb{Z}[x]\))? \emph{Discuss according to the parity of \(a\)}.
		
		\item Show that \( \gcd(x^a - 1, x^b - 1) = x^d - 1 \) where \( d = \gcd(a, b) \). \emph{Hint: link an elementary step of the Euclidean algorithm on polynomials to one step on the integers.}
		
		\item In this question, assume that \(a\) and \(b\) are coprime.
		\begin{enumerate}
			\item What is \( \gcd(x^{2a} - 1, x^{2b} - 1) \)?
			
			\item Using Bézout's identity, show that \( x + 1 \) belongs to the ideal \( \langle A(x), B(x) \rangle \), i.e., there exist \( U(x), V(x) \in \mathbb{Z}[x] \) such that \( x + 1 = (x^a + 1)U(x) + (x^b + 1)V(x) \).
			
			\item Show that if \(a\) and \(b\) are odd, the gcd of \( A(x) \) and \( B(x) \) is \( x + 1 \).
			
			\item Show that if \(a\) or \(b\) is even, \( A(x) \) and \( B(x) \) are coprime and that \( 2 \) belongs to the ideal \( \langle A(x), B(x) \rangle \).
			
			\item Deduce that if \( A(x) \) and \( B(x) \) are coprime, then \( \gcd(A(n), B(n)) \) is \( 1 \) or \( 2 \). \emph{Hint: see the beginning of Section \ref{sec:delta}.}
			
			\item Conclude for the case where \( a \) and \( b \) are coprime.
			
		\end{enumerate}
		
		\item Now, do not assume \(a\) and \(b\) are coprime. Deduce from the previous question that if \( A(x) \) and \( B(x) \) are coprime polynomials, then the sequence \( \gcd(A(n), B(n)) \) is a periodic sequence with pattern \([1, 2]\).
		
	\end{enumerate}
\end{exercise}

\section*{Perspective}

Let's conclude with examples of polynomials having multiple roots after reduction modulo \( p \), and therefore for which Hensel's lemma no longer applies, Proposition \ref{prop:delta} is no longer valid, and the role of \( \delta \) cannot be as direct as in this proposition.
Let \( A(x) = x^2 + 27 \) and \( B(x) = x^2 - 18x + 108 \). These are two coprime polynomials with \( \Delta = 3^7 \times 7 \) and \( \delta = 3^5 \times 7 \).
The sequence of terms \( A(n) \wedge B(n) \wedge 3^{\infty} \) is periodic with pattern \( [27,1,1,9,1,1,9,1,1] \) of length $9$.
Contrary to what happens in Proposition \ref{prop:delta}, the power of 3 appearing in \( \delta \), namely \( 3^5 \), is greater than the length 9 of the pattern or its greatest value \( 27 \).

To broaden the perspective, polynomials over the ring \( \Zz/n\Zz \) sometimes exhibit surprising behavior.
For example, \( A(x) = (x+1)(x+7) \) and \( B(x) = (x+3)(x+5) \) are coprime polynomials (with \( \Delta = 64 \) and \( \delta = 8 \)). The sequence of terms \( A(n) \wedge B(n) \) is periodic with pattern \( [1,8] \).
The polynomials \( A \) and \( B \) are equal modulo $2$, with $1$ as a double root modulo $2$; they are also equal modulo $4$.
More surprisingly, since \( (x+1)(x+7) \equiv (x+3)(x+5) \pmod{8} \), the polynomials \( A \) and \( B \) are equal modulo $8$. This is possible because \( \Zz/8\Zz[x] \) is not a factorial ring.
There is still much to discover!

\bigskip

\emph{Acknowledgements.} We thank the referees for their comments and especially for suggesting Exercise \ref{ex:exo}.


\bibliographystyle{plain}
\bibliography{xgcd.bib}

\end{document}